\documentclass{amsart}
\usepackage{amssymb,amsmath,latexsym,amsthm}
\usepackage[english]{babel}

\usepackage{amsmath,amssymb,amsbsy,amsfonts,amsthm,latexsym,
                       amsopn,amstext,amsxtra,euscript,amscd}

\usepackage{array}
\usepackage{ifthen}
\usepackage{url}
\usepackage{hyperref}

\newtheorem{theorem}{Theorem}[section]
\newtheorem{lemma}{Lemma}[section]

\newtheorem{conjecture}{Conjecture}[section]
\theoremstyle{definition}

\newtheorem*{remark}{Remark}

\newtheorem{question}{Question}[section]

\def\hh{{\rm h}\kern.4pt}

\date{\today}

\author{Nikhil Balaji}

\address[Nikhil Balaji]{Department of Computer Science and Engineering \\
	Indian Institute of Technology Delhi \\
	New Delhi 110016 \\
	India}

\author{Florian Luca}
\address[Florian Luca]{School of Maths	Wits University \\
	1 Jan Smuts, Braamfontein\\
	Johannesburg 2000\\
	South Africa}
\address{
Research Group in Algebric Structures and Applications \\
King Abdulaziz University\\ 
Abdulah Sulayman\\ 
Jeddah 22254\\
Saudi Arabia}

\address{Centro de Ciencias Matem\'{a}ticas UNAM \\ 
Morelia \\
Mexico \\
}
\email{nbalaji@cse.iitd.ac.in, florian.luca@wits.ac.za}
\thanks{F. L. worked on this paper while visiting the Max Planck Institute for Software Systems in Saarbr\"ucken, Germany in Fall of 2020. This author thanks the Institute for hospitality and support}

\begin{document}

\title[Terms of Lucas sequences having a smooth divisor]{Terms of Lucas sequences having a large smooth divisor}

\pagenumbering{arabic}

\maketitle

\begin{abstract}
We show that the $Kn$--smooth part of $a^n-1$ for an integer $a>1$ is $a^{o(n)}$ for most positive integers $n$.
\end{abstract}

\section{Introduction}


It is known that if for every $n$, the sequence $\binom{2n}{n}$ can be computed in $O(\log^k{n})$ arithmetic operations for a fixed constant $k$, then integers can be factored efficiently~\cite{Lipton, Shamir}. We ask if there exist linearly recurrent sequences which contain many small factors like $\binom{2n}{n}$. If such sequences exist, they can be used instead of $\binom{2n}{n}$ to factor integers. This is because the $n$-th term of any linearly recurrent sequence can be computed in $O(\log{n})$ arithmetic operations using repeated squaring of the companion matrix~\cite{BM}. We first set up some notation to formally state our question. 


Let $P(n)$ be the largest prime factor of $n$ and $s_y(n)$ be the largest divisor $d$ of $n$ with $P(d)\le y$. Thus, $s_y(n)$ is the {\it $y$-smooth} part of $n$. Given a sequence 
${\bf u}=(u_n)_{n\ge 0}$ of positive integers we ask whether we can find $c>1$ and $K$ such that
$$
{\mathcal A}_{K,c,{\bf u}}=\{n: s_{Kn}(u_n)>c^n\}
$$ 
contains many elements. For example, if $u_n=\binom{2n}{n}$ is the sequence of middle binomial coefficients, then ${\mathcal A}_{2,2,{\bf u}}$ contains all the positive integers. The main question we tackle in this paper can be formally stated as follows.

\begin{question}
Does there exist a linearly recurrent sequence {\bf u} such that ${\mathcal A}_{K,c,{\bf u}}$ is infinite?
\end{question} 

Here we address the problem in the simplest case namely
$u_n=a^n-1$ for some positive integer $a$. Our results are easily extendable to all Lucas sequences, in particular, the sequence of Fibonacci numbers.

To start we recall the famous $ABC$-conjecture. Put
$$
{\text{\rm rad}}(n)=\prod_{p\mid n} p
$$
for the algebraic radical of $n$.

\begin{conjecture}
For all $\varepsilon>0$ there exists a constant $K_{\varepsilon}$ such that whenever $A,B,C$ are coprime nonzero integers with $A+B=C$, then 
$$
\max\{|A|,|B|,|C|\}\le K_{\varepsilon}{\text{\rm rad}}(ABC)^{1+\varepsilon}.
$$
\end{conjecture} 

Throughout this paper $a>1$ is an integer and $u_n=a^n-1$. 

We have the following result.

\begin{theorem}
\label{thm:1}
Assume the $ABC$ conjecture. Then for any $K>0$, $c>1$, the set ${\mathcal A}_{K,c,{\bf u}}$ is finite.
\end{theorem}

One can ask what can one prove unconditionally. Maybe we cannot prove that ${\mathcal A}_{K,c,{\bf u}}$ is finite but maybe we can prove that it is {\it thin}, that is that 
it does not contain too many integers. This is the content of the next theorem.

\begin{theorem}
\label{thm:2}
We have
\begin{equation}
\label{eq:3}
\#({\mathcal A}_{K,c,{\bf u}}\cap [1,N])\ll N\exp\left(-\frac{\log N}{156\log\log N}\right).
\end{equation}
\end{theorem}

In particular, if one wants to find for all large $N$ an interval starting at $N$ of length $k$, that is $[N+1,\ldots,N+k]$ which has non-empty intersection with ${\mathcal A}_{K,c,{\bf u}}$
then infinitely often one should take $k>\exp(\log N/(157\log\log N))$. But if the $ABC$ conjecture is true, one will no longer find elements of ${\mathcal A}_{K,c,{\bf u}}$ in the above interval for large $N$ no matter how large $k$ is.

\section{Proofs}

\subsection{The proof of Theorem \ref{thm:1}}

We apply the $ABC$ conjecture to the equation
$$
a^n-1=st,\quad s:=s_{Kn}(u_n),\quad t=(a^n-1)/s
$$
for $n\in {\mathcal A}_{K,c,{\bf u}}$ with the obvious choices. Note that 
$$
{\text{\rm rad}}(s)=\prod_{\substack{p\le Kn\\ p\mid a^n-1}} p\quad {\text{\rm and}}\quad  t<(a/c)^n.
$$ 
We then have
$$
a^n\ll_{\varepsilon} (a\cdot {\text{\rm rad}}(s) t))^{1+\varepsilon}\ll \left(\prod_{\substack{p\le Kn\\ p\mid a^n-1}} p\right)^{1+\varepsilon} (a/c)^{n(1+\varepsilon)}.
$$
We may of course assume that $1<c<a$. Then
$$
\sum_{\substack{p\le Kn\\ p\mid a^n-1}} \log p\ge \frac{n}{1+\varepsilon}(\log a-(1+\varepsilon)\log(a/c))+O_{\varepsilon}(1).
$$
We choose $\varepsilon>0$ small enough so that $\log a-(1+\varepsilon)\log(a/c)>0$. Then, we get
\begin{equation}
\label{eq:SaK}
S_{a,K}(n):=\sum_{\substack{p\le Kn\\ p\mid a^n-1}} \log p\gg_{\varepsilon} n.
\end{equation}
The next lemma shows that the left--hand side above is $\le n^{2/3+o(1)}$ as $n\to\infty$. This is unconditional and finishes the proof of Theorem \ref{thm:1}.

\begin{lemma}
\label{lem:1}
We have 
$$
S_{K,a}(n)\le K^{1/2} n^{1/2+o(1)}
$$
as $n\to \infty$.
\end{lemma}

\begin{proof}
Let $\ell_p$ be the order of $a$ modulo $p$; that is the smallest positive integer $k$ such that $a^k\equiv 1\pmod p$. Since primes $p$ participating in $S_{K,a}(n)$ 
have $p\mid a^n-1$, it follows that $\ell_p\mid n$. Since also such primes are $O(n)$, it follows that
$$
S_{K,a}\ll \#P_{K,n}\log n,
$$
where $P_{K,a}(n):=\{p\le Kn: \ell_p\mid n\}$. To estimate $P_{K,a}(n)$ we fix a divisor $d$ of $n$ and look at primes $p\le Kn$ such that $\ell_p=d$. Such primes $p$ have the property that $p\equiv 1\pmod d$ by Fermat's Little Theorem. In particular, the number of such (without using results on primes in progressions) is at most
$$
\left\lfloor \frac{Kn}{d}\right\rfloor\le \frac{Kn}{d}.
$$
However, since these primes divide $a^d-1$, the number of them is $O(d)$. Thus, for a fixed $d$ the number of such primes is 
$$
\ll \min\left\{\frac{Kn}{d},d\right\}\ll  (Kn)^{1/2}.
$$
Summing this up over all divisors $d$ of $n$ we get that 
$$
\#P_{K,a}(n)\ll d(n) (Kn)^{1/2}\le K^{1/2} n^{1/2+o(1)}
$$
as $n\to\infty$, where we used $d(n)$ for the number of divisors of $n$ and the well-known estimate $d(n)=n^{o(1)}$ as $n\to\infty$ (see Theorem 315 in \cite{HW})). Hence, 
$$
S_{K,a}(n)\ll \#P_{K,a}(n)\log n\le K^{1/2} n^{1/2+o(1)}
$$
as $n\to \infty$, which is what we wanted. 
\end{proof}

\begin{remark}  The current Lemma \ref{lem:1} was supplied by the referee. Our initial statement was weaker. 
The combination between Lemma \ref{lem:1} and estimate \eqref{eq:SaK} shows that we can even take $K$ growing with $n$ such as $K=n^{1-\varepsilon}$ in the hypothesis of Theorem \ref{thm:1} 
and retain its conclusion. This has been also noticed in \cite{MuWo}. 
\end{remark}

\subsection{The proof of Theorem \ref{thm:2}} It is enough to prove an upper bound comparable to the upper bound from the right--hand side of \eqref{eq:3} for 
$\#({\mathcal A}_{K,c,{\bf u}}\cap (N/2,N])$ as then we can replace $N$ by $N/2$, then $N/4$, etc. and sum up the resulting inequalities. So, assume that 
$n\in (N/2,N]$. We estimate
$$
Q_N:=\prod_{n\in (N/2,N]} s_{KN}(u_n).
$$
On the one hand, since $s_{KN}(u_n)\ge s_{Kn}(u_n)\ge c^{n}\ge c^{N/2}$ for all $n\in {\mathcal A}_{K,c,{\bf u}}$, we get that 
$$
\log Q_N\gg N(\#{\mathcal A}_{K,c,{\bf u}}\cap (N/2,N]).
$$
Next, writing $\nu_p(m)$ for the exponent of $p$ in the factorisation of $m$, we have
\begin{equation}
\label{eq:logQ}
\log Q_N=\sum_{n\in (N/2,N]} \sum_{p\le KN}  \nu_p(u_n)\log p\le \sum_{p\le KN}\log p \sum_{n\in (N/2,N]} \nu_p(u_n).
\end{equation}
Let $o_p:=\nu_p(u_{\ell_p})$. It is well-known that if $p$ is odd then
$$
\nu_p(u_n)=\left\{\begin{matrix} o_p+\nu_p(n) & {\text{\rm if}} & \ell_p\mid n;\\ 0 & {\text{\rm otherwise}} & ~ \\
\end{matrix}\right.
$$
(see, for example, (66) in \cite{Ste}). 
In particular, if $p\mid u_n$, then $p^{o_p}\mid u_n$. Furthermore, for each $k\ge 0$, the exact power of $p$ in $u_n$ is $o_p+k$ if and only if $\ell_p p^k$ divides $n$ and 
$\ell_p p^{k+1}$ does not divide $n$. When $p=2$, we may assume that $a$ is odd (otherwise $\nu_2(u_n)=0$ for all $n\ge 1$), and the right--hand side of the above forrnula needs to be ammended 
to 
$$
\nu_2(u_n) =\left\{\begin{matrix} o_2 & {\text{\rm if}} & 2\nmid n;\\
o_p+\nu_2(a+1)+\nu_2(n/2) & {\text{\rm if}} & 2\mid n.\\
\end{matrix}\right.
$$
Thus, for odd $p$, 
\begin{equation}
\label{eq:x}
\sum_{n\in (N/2,N]}\nu_p(u_n)=o(p)\#\{N/2<n\le N: \ell_p\mid n\}+\sum_{k\ge 1} \#\{N/2<n\le N: \ell_p p^k\mid n\}.
\end{equation}
A similar formula holds for $p=2$. In particular, for $p=2$, we have 
$$
\sum_{n\in (N/2,N]} \nu_2(u_n)=O(N).
$$
Thus, the prime $p=2$ contributes a summand of size $O(N)$ to the right--hand side of \eqref{eq:logQ}. From now on, we assume that $p$ is odd. The first cardinality in the right--hand side of formula \eqref{eq:x} above is 
$$
\#\{N/2<n\le N: \ell_p\mid n\} \le \left\lfloor \frac{N}{2\ell_p}\right\rfloor+1\ll \frac{N}{\ell_p}.
$$
The remaining cardinalities on the right--above can be bounded as 
$$
 \#\{N/2<n\le N: \ell_p p^k\mid n\}\le \left \lfloor \frac{N}{2\ell_p p^k}\right\rfloor+1\ll \frac{N}{\ell_p p^k}.
 $$
Thus,
$$
\sum_{n\in (N/2,N]}\nu_p(u_n)\ll \frac{N o_p}{\ell_p}+\sum_{k\ge 1} \frac{N}{\ell_p p^k}\ll \frac{N o_p}{\ell_p}+\frac{N}{\ell_p p}.
$$
We thus get
$$
\log Q_N\ll N\sum_{p\le Kn} \frac{o_p \log p}{\ell_p}+N\sum_{p\le Kn} \frac{\log p}{\ell_p p} \ll N\sum_{p\le Kn} \frac{o_p \log p}{\ell_p} :=S.
$$
It remains to bound $S$. Since $p^{o_p}\mid a^{\ell_p}-1$, we get that $p^{o_p}<a^{\ell_p}$ so $o_p\log p\ll \ell_p$. Hence, 
$$
S=N\sum_{p\le KN} \frac{o_p\log p}{\ell_p}\ll N\pi(KN)\ll_K \frac{N^2}{\log N}.
$$
We get the first non--trivial upper bound on $\#({\mathcal A}_{K,c,{\bf u}}\cap (N/2,N]$, namely
$$
N\#({\mathcal A}_{K,c,{\bf u}}\cap (N/2,N])\ll \log Q_N\ll S \ll \frac{N^2}{\log N}+N\log\log N\ll_K \frac{N^2}{\log N},
$$
so 
$$
 \#({\mathcal A}_{K,c,{\bf u}}\cap (N/2,N])\ll_K \frac{N}{\log N}.
$$
To do better, we need to look more closely at $o_p\log p/\ell_p$ for primes $p\le KN$. We split the sum $S$ over primes $p\le KN$ in two subsums. The first is over the primes in the set 
$Q_1$ consisting of $p$ such that $o_p\log p/\ell_p<1/y_N$, where $y_N$ is some function of $N$ which we will determine later. We let $Q_2$ be the complement of $Q_1$ in the set of primes $p\le Kn$. The sum over primes 
$p\in Q_1$ is 
$$
S_1=N\sum_{p\in Q_1}\frac{o_p\log p}{\ell_p}\le \frac{N}{y_N}\pi(KN)\ll_K \frac{N^2}{y_N\log N}.
$$
For $Q_2$, we use the trivial estimate
$$
S_2=N\sum_{p\in Q_2}\frac{o_p\log p}{\ell_p}\ll N\#Q_2,
$$
and it remains to estimate the cardinality of $Q_2$. Note that $Q_2$ consists of primes $p$ such that $o_p>\ell_p/(y_N\log p)\gg \ell_p/(y_N\log N)$. 
We put $\ell_p$ in dyadic intervals. That is $\ell_p\in (2^i,2^{i+1}]$ for some $i\ge 0$. Then primes $p\le KN$ in $Q_2$ with such $\ell_p$ have the property that 
$o_p\gg 2^i/(y_N\log N)$. Hence,
\begin{eqnarray*}
\frac{2^i\#(Q_2\cap (2^i,2^{i+1}])}{y_N \log N} & \ll & \sum_{p\in Q_2\cap (2^i,2^{i+1}]} \nu_p(a^{\ell_p}-1)\log p\le \sum_{\ell \in (2^i,2^{i+1}]} \log(a^{\ell}-1)\\
& \ll &  \sum_{\ell\in (2^i,2^{i+1}]}\ell \ll 2^{2i},
\end{eqnarray*}
which gives 
$$
\#(Q_2\cap (2^i,2^{i+1}])\ll 2^i y_N\log N.
$$
Summing up over all the $i$, we get 
$$
\#Q_2\le 2^I y_N\log N,
$$
where $I$ is maximal such that $(2^I,2^{I+1}]$ contains an element $p$ of $Q_2$. By a result of Stewart (see Lemma 4.3 in \cite{Ste}),
\begin{eqnarray*}
2^I & < & \ell_p<o_p y_N\log N<p\exp\left(-\frac{\log p}{51.9\log\log p}\right)y_N \log N \log \ell_p\\
& \ll & KN\exp\left(-\frac{\log(Kn)}{51.9\log\log(KN)}\right)y_N\log(KN)^2\\
& \ll_K &  
N\exp\left(-\frac{\log N}{51.95\log\log N}\right) y_N(\log N)^2.
\end{eqnarray*}
Thus,
\begin{eqnarray*}
\#Q_2 & \ll & 2^I y_N\log N\ll_K  N\exp\left(-\frac{\log N}{51.95\log\log N}\right) y_N^2(\log N)^3\\
& \ll & N\exp\left(-\frac{\log N}{52\log\log N}\right) y_N^2.
\end{eqnarray*}
Choosing $y_N:={\displaystyle{\exp\left(c\frac{\log N}{\log\log N}\right)}}$ with a positive constant $c$ to be determined later, we get 
\begin{eqnarray*}
N\#({\mathcal A}_{K,c,{\bf u}}\cap (N/2,N]) & \ll &  N\#Q_2+\frac{N}{y_N\log N}\\
& \ll _K & N\left(\exp\left(\left(2c-\frac{1}{52}\right)\frac{\log N}{\log\log N}\right)+\exp\left(-\frac{c\log N}{\log\log N}\right)\right).
\end{eqnarray*}
Choosing $c:=1/156$, we get 
$$
\#({\mathcal A}_{K,c,{\bf u}}\cap(N/2,N])\ll N\log N \exp\left(-\frac{\log N}{156\log\log N}\right),
$$
which is what we wanted. 

\subsection*{Acknowledgements}

We thank the referee for suggesting the current Lemma \ref{lem:1} with its proof and for pointing out reference \cite{MuWo}.

\end{document}